\newif\ifsiam
\newif\ifnummat

 \siamtrue

 \nummatfalse

\ifsiam
 \documentclass[final,leqno]{siamltex}
\else
  \ifnummat
     \documentclass[nummat]{svjour}
  \else
     \documentclass[twocolumn,fleqn,runningheads,final]{svjour2}
     \journalname{Computing and Visualization in Science}
     \usepackage{mathptmx}
  \fi

\fi

\usepackage{curves,amsmath,amssymb,color,fancybox,exscale}

\definecolor{dark}{gray}{0.6}
\definecolor{light}{gray}{0.8}

\def\eref#1{(\ref{#1})}

\def\Re{\mathbb R}

\def\ip#1#2{( #1 , #2 )}
\def\bigip#1#2{\big( #1 , #2 \big)}

\def\Bigip#1#2{\Big( #1 , #2 \Big)}

\def\bip#1#2{\langle #1 , #2 \rangle}
\def\bigbip#1#2{\big\langle #1 , #2 \big\rangle}

\def\norm#1{|\!| #1 |\!|}
\def\bignorm#1{\big|\!\big| #1 \big|\!\big|}

\def\snorm#1{| #1 |}

\def\enorm#1{|\!|\!| #1 |\!|\!|}

\def\rebAuthor{Randolph E. Bank}

\def\mmAuthor{Maximillan S. Metti}

\ifsiam
    \def\rebShortAuthor{R.~E.~Bank}
    \def\mmShortAuthor{M.~S.~Metti}
\else
    \def\rebShortAuthor{Bank}
    \def\mmShortAuthor{Metti}
\fi

\def\rebAddress{Department of Mathematics, University of California, San Diego,
 La Jolla, California 92093-0112. Email: rbank@ucsd.edu.}

\def\mmAddress{Department of Mathematics, University of California, San Diego,
 La Jolla, California 92093-0112. Email: mmetti@ucsd.edu.}

\def\rebThanks{The work of this author was supported by the National
Science Foundation under  contract DMS-1318480.}

\def\mmThanks{The work of this author was supported by the National
Science Foundation under  contract DMS-1318480.}

\title{An Error Analysis of Some Higher Order Space-Time Moving Finite Elements}

\def\shortTitle{Higher Order Moving Finite Elements}

\def\myKeywords{Moving Finite Elements, Error Analysis, Symmetric Error Estimate, Convection-Dominated}

\def\myAMS{65M15, 65M50, 65M60}

\def\myAbstract{
This is a study of certain finite element methods designed for convection-dominated, time-dependent partial differential equations.
Specifically, we analyze high order 
space-time tensor product finite element discretizations,
used in a method of lines approach coupled with mesh modification
to solve linear partial  differential equations. Mesh modification can be both continuous
(moving meshes) and discrete (static rezone).
These methods can lead to significant savings in computation costs for problems having solutions that develop steep moving fronts or
other localized time-dependent features of interest.
Our main result is a symmetric a priori error estimate for the finite element solution computed in this setting.
}

\begin{document}


\ifsiam
  \author{\rebAuthor%
         \thanks{\rebAddress\ \rebThanks}
          \and
          \mmAuthor%
         \thanks{\mmAddress\ \mmThanks}
         }
  \maketitle

  \begin{abstract}\myAbstract\end{abstract}
  \pagestyle{myheadings}
  \thispagestyle{plain}
  \markboth{\rebShortAuthor\ and \mmShortAuthor }{\shortTitle}

\else


  \ifnummat
      \author{\rebAuthor%
             \thanks{\rebThanks}
             \and
             \mmAuthor%
             \thanks{\mmThanks}
              }
  \else
      \author{\rebAuthor%
             \thanks{\rebShortAuthor : \rebThanks}
             \and
             \mmAuthor%
             \thanks{\mmShortAuthor : \mmThanks}
              }
  \fi
  \institute{\rebShortAuthor : \rebAddress \\ \mmShortAuthor : \mmAddress}
  \date{Received: \today\  / Accepted: date}
  \maketitle
  \begin{abstract}\myAbstract\end{abstract}
  \begin{subclass}\myAMS \end{subclass}
\fi



\section{Introduction}\label{sec1}

Computing accurate solutions to convection-dominated partial differential equations using standard finite element methods can be computationally expensive,
and sometimes prohibitively so.
Consequently, the use of adaptive methods can lead to great savings in computation time and maintain accuracy of the computed solution 
\cite{BABUSKADORR81,BANKNGUYEN}.
It is often the case that regions in which the solution to a partial differential equation is rough or rapidly changing are relatively small compared to the overall domain
and adaptive methods leverage this fact by focusing more computational effort by placing a higher concentration of the degrees of freedom in these regions
and avoiding ``over-solving'' where the solution is smooth \cite{BANKSHERMANWEISER}.
Effectively, adaptive methods are designed to automate the process of finding a finite element space that is well-suited to solving a given differential equation.
For time-dependent problems, these critical regions can move throughout the domain, as in the case of steep moving fronts,
and short time steps may be required to maintained a desired level of accuracy in these regions
\cite{HUNDSDORFERVERWER,MILLER1}.
In order to avoid these short time steps,  a moving mesh can be used to continuously track this moving region.

Moving finite elements were initially proposed by Miller and Miller in \cite{MILLER1,MILLER2}
and have been analyzed and implemented in the context of linear and nonlinear problems.
One of the main themes in moving finite element methods has been to devise strategies for effectively moving the mesh efficiently and accurately solve a given problem.
In this paper, however, we focus on providing an error analysis for a space-time moving finite element method that allows for a variety of mesh motion schemes.

The first error analysis of moving finite element methods is given in Dupont \cite{DUPONT82}, where a symmetric error bound of the form
\begin{equation}        \label{quasi-optimality}
        \enorm{u-u_h}   \le     C \inf_{v \in \mathcal{V}_h} \enorm{u-v},
\end{equation}
was proved in the semi discrete case (continuous in time, discrete in space).
Here $u$ and $u_h$ are the true solution and the finite element solution to the differential equation, respectively, $\mathcal{V}_h$ is the finite element space,
and $\enorm{\cdot}$ is a specially defined mesh dependent energy norm related to the differential equation.
Symmetric error bounds are proven for linear moving finite elements by Bank and Santos in \cite{BANKSANTOS,SANTOSTHESIS}
and by Dupont and Mogultay \cite{DUPONTMOGULTAY}.
Some symmetric error estimates for mixed methods that use moving meshes are proven in \cite{LIU_ETAL}.

Here we prove a symmetric error bound like \eref{quasi-optimality} for finite element spaces of arbitrary order,
and more general time integration schemes.
To do this, we first describe a space-time tensor-product finite element space that allows the spatial discretization to evolve continuously in time,
except at discrete time steps where the mesh is allowed to reconfigure in a discontinuous fashion.
Since this analysis employs higher order finite element spaces, say order $p$, 
the spatial nodes are permitted to follow piecewise polynomial trajectories of degree $p$ time, allowing smoother and more dynamic mesh motion.

In \S \ref{sec2}, the time-dependent linear convection-diffusion equation is introduced and the construction of a space-time tensor-product finite element space is described.
In \S \ref{sec3}, the notation for the analysis is established, a new space-time shape regularity constraint is proposed for the moving finite elements,
and some preliminary results are given.
A space-time moving finite element method is proposed and analyzed in \S \ref{sec4},
and a symmetric error estimate is proven for finite element spaces of arbitrary order.
We conclude in \S \ref{sec5} with some remarks on this error analysis and an error analysis for which more general time integration schemes can be employed.

\section{Notation and Definitions}\label{sec2}

The spatial domain $\Omega$ is assumed to be a compact subset of $\Re^d$,
where {$d=1,2,$ or $3$}, with boundary $\partial \Omega$.
The time domain is a finite interval, $(0, T]$, and the space-time 
domain is given by $\mathcal{F} \equiv \Omega \times (0, T]$.

Let $a$, $b$, $c$, and $f$ be smooth and bounded functions defined on $\mathcal{F}$
such that there exist constants $\bar{a}>0$ and $\bar{c}\ge0$ with $a \ge \bar{a}$ and $c \ge \bar{c}$,
and let $g$ be piecewise continuous on $\partial \Omega$.
Let $u_0$ be a given initial condition for the solution on $\mathcal{F}$ and let $n$ denote the outward unit normal vector to the boundary $\partial \Omega$.
The solution to the differential equation, denoted by $u$, is the function that satisfies 
\begin{align}
        u_t     -       \nabla \cdot ( a \nabla u )     +       b \cdot \nabla u        +       c u     &=      f,              &\mathrm{in}\ \mathcal{F},
                                                                                                                                                        \label{strong pde}\\
                                                                                a  \nabla u \cdot n     &=      g,              &\mathrm{on}\ \partial \Omega \times (0, T],
                                                                                                                                                        \label{strong bc}\\
                                                                                                u(x, 0) &=      u_0(x), &\mathrm{for}\ x\ \mathrm{in}\ \Omega.\nonumber
\end{align}
For convenience, a Neumann boundary condition \eref{strong bc} is assumed;
our results still hold with minor changes when other boundary conditions are imposed.

When the convection velocity, $b$ in \eref{strong pde}, is large, 
the solution of the equation may develop steep {shock layers} that sweep through the spatial domain.
Such moving structures can be difficult to track accurately and require small time steps for non-moving finite elements.
To offset the effects of a potentially large convection velocity,
we replace the time derivative in \eref{strong pde} with space-time directional derivative, as in the method of characteristics.
Define a time-dependent parameterization of the spatial variable, $x(t)$,
and the \emph{characteristic} derivative as
\[
        u_\tau(x(t),t)  \equiv  \frac{d}{dt} u(x(t),t)  =       u_t (x(t),t) + x_t \cdot \nabla u(x(t),t).
\]

We replace the time derivative with the characteristic derivative  in \eref{strong pde}
and propose a weak form of the differential equation that can reduce the presence of the convection velocity:
find $u$ with $u(t) \in \mathcal{H}^1(\Omega)$ and $u_t(t) \in \mathcal{L}_2(\Omega)$
such that for all $\chi$ in $\mathcal{H}^1(\Omega)$ and $0 < t \le T$,
\begin{align}
                        \bigip{ u_\tau(\cdot,t)}{\chi}  +       \mathcal{A}_\tau\ip{t;  u}{\chi}        &=      \bigip{f(\cdot,t)}{\chi}        +       \bigbip{g(\cdot,t)}{\chi},      
                                                                                                                                                                        \label{variational form diff eqn} \\
\intertext{and when $t=0$}
                                                                                \bigip{u(\cdot,0)}{\chi}        &=      \bigip{u_0}{\chi}.      \nonumber
\end{align}
The inner-products are given by
\begin{align*}
        \ip{f}{\chi}                                    &=      \int_{\Omega} f(x) \chi(x) \ dx,                        \\
        \bip{g}{\chi}                           &=      \int_{\partial \Omega} g(s) \chi(s)     \ ds,
\end{align*}
and define the time-dependent bilinear form
\begin{multline*}
        \mathcal{A}_\tau\ip{t; u}{\chi} \equiv  \int_{\Omega}   a(x,t) \nabla u(x,t) \cdot \nabla \chi(x)       +       (b(x,t)-x_t(t)) \cdot \nabla u(x,t) \ \chi(x)   \\
                                                                +       c(x,t) u(x,t) \chi(x)   \ dx.
\end{multline*}
Notice that parameterizing the spatial variable so that $x_t \approx b$ leads to a formulation where the convection velocity is much less prominent,
as it is ``absorbed'' into the characteristic derivative.

\subsection{A space-time tensor-product moving finite element space}

To compute a solution to the differential equation,
we restrict the trial and test spaces of equation \eref{variational form diff eqn} to a finite element space, $\mathcal{V}_h^p$,
where $p\ge1$ is the maximum order of the piecewise polynomials that generate the space.
The finite element spaces described in this paper are tensor-products of finite element spaces on $\Omega$
with a finite element discretization of the time domain.
These are akin to the finite element spaces described in \cite{BANKSANTOS,DUPONTMOGULTAY,SANTOSTHESIS},
except that these previous works restricted attention to the case of linear elements ($p=1$).

Partition the time domain into $m$ disjoint intervals, where the endpoints of the partitions are given by $\{t_i\}$ and satisfy
\[
        0       =       t_0     <       t_1     <       \ldots  <       t_m     =       T,
\]
with $\Delta t_i \equiv t_i - t_{i-1}$ for $i = 1, \ldots, m$.
Within each time partition, we define a triangulation of the spatial domain $\Omega$ with nodes
$\{ x_k(t) \}_{k=0}^{n_i}$
such that each node is a polynomial of degree $p$ on $( t_{i-1}, t_i ]$.
To avoid mesh degeneration, the mesh topology is required not to change in time
--- when $d=1$, this corresponds to $x_{k-1}(t) < x_k(t)$ for $t_{i-1} \le t \le t_i$ and $k=1,\ldots, n_i$.
The resulting mesh is a series of time partitions on which we have a triangulation of the mesh that evolves continuously in time.
An example time partition, with quadratic mesh motion ($p=2$) and $d=1$, is depicted in figure \ref{fig: space time partition}.

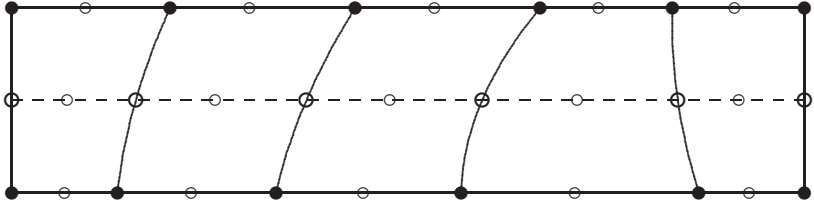
\begin{figure}[h!btp] 
        \begin{picture}(300,100)(-35,0)
                \put(0,0){\thicklines{\line(0,1){70}}}
                \put(0,0){\thicklines{\line(1,0){300}}}
                \put(0,35){\dashbox{4}(300,0){}}
                \put(0,70){\thicklines{\line(1,0){300}}}
                \put(300,0){\thicklines{\line(0,1){70}}}
                \put(0,0){\thicklines{\circle*{5}}}
                \put(300,70){\thicklines{\circle*{5}}}
                \put(300,0){\thicklines{\circle*{5}}}
                \put(0,70){\thicklines{\circle*{5}}}
                \put(250,70){\thicklines{\circle*{5}}}
                \put(200,70){\thicklines{\circle*{5}}}
                \put(130,70){\thicklines{\circle*{5}}}
                \put(60,70){\thicklines{\circle*{5}}}
                \put(40,0){\thicklines{\circle*{5}}}
                \put(100,0){\thicklines{\circle*{5}}}
                \put(170,0){\thicklines{\circle*{5}}}
                \put(260,0){\thicklines{\circle*{5}}}
                \put(47,35){\thicklines{\circle{5}}}
                \put(111.5,35){\thicklines{\circle{5}}}
                \put(178,35){\thicklines{\circle{5}}}
                \put(252,35){\thicklines{\circle{5}}}
                \put(0,35){\thicklines{\circle{5}}}
                \put(300,35){\thicklines{\circle{5}}}
                \put(20,0){\circle{4}}
                \put(68,0){\circle{4}}
                \put(133,0){\circle{4}}
                \put(213,0){\circle{4}}
                \put(279,0){\circle{4}}
                \put(21,35){\circle{4}}
                \put(77,35){\circle{4}}
                \put(143,35){\circle{4}}
                \put(214,35){\circle{4}}
                \put(275,35){\circle{4}}
                \put(28,70){\circle{4}}
                \put(90,70){\circle{4}}
                \put(160,70){\circle{4}}
                \put(222,70){\circle{4}}
                \put(273.5,70){\circle{4}}
                \qbezier(40,0)(44,35)(60,70)
                \qbezier(100,0)(108,35)(130,70)
                \qbezier(170,0)(171,35)(200,70)
                \qbezier(260,0)(249,35)(250,70)
        \end{picture}
\caption{An example space-time mesh partition with $d=1$ and $p=2$.
The filled circles represent the space-time ``hat'' basis functions;
hollow circles correspond to basis functions that are the product of a ``bump'' function with a ``hat'' or
``bump'' function.}
\label{fig: space time partition}
\end{figure}


As described above, a continuously moving mesh is defined on each time partition.
Between time partitions, however, discontinuous changes in the mesh are permitted.
These discontinuities provide for the periodic addition, deletion, and relocation of the spatial nodes to track structures that may develop or vanish over time.
Another important benefit of these mesh reconfigurations is that they provide a means to avoid the nodes from entangling or colliding
\cite{CARLSONMILLER1,CARLSONMILLER2,KUPRATTHESIS}.

For $d=1$, the elements in the mesh are curvilinear trapezoids with flat parallel edges corresponding to the beginning and end of the time partitions,
and curved edges representing the moving space nodes --- recall that these edges are polynomial curves of degree $p$.
The associated space-time reference element is the Cartesian product of the reference element in space with the reference element in time.
Moreover, the space-time basis functions are given by the tensor products of the degree $p$ polynomial spatial basis functions on $[0,1]$
and the degree $p$ polynomial temporal basis functions on $[0,1]$.
This implies that there are $(p+1)^2$ degrees of freedom associated with the reference element.
Since a tensor product is used to define the space-time basis functions,
the degrees of freedom on the reference element are aligned into time slices.
The degrees of freedom are represented by the filled and empty circles in figure {\ref{fig: space time partition}} and their alignment
into time slices is emphasized by the dashed line.

When $d=2$, the elements are curvilinear triangular prisms
with flat triangular parallel edges corresponding to the beginning and end of a time partition,
and the curved edges corresponding to the space nodes moving along polynomial trajectories of degree $p$.
The reference element is the Cartesian product of the unit triangle and the unit interval, giving a wedge-like shape,
and has $(p+1)\times(p+1)(p+2)/2$ degrees of freedom.
The space-time basis functions are the tensor product of the degree $p$ polynomial basis functions in space, on the unit triangle,
and the degree $p$ polynomial basis functions in time on the unit interval.

For $d=3$, the space-time reference element is the Cartesian product of the unit tetrahedron with the unit interval reference element for time
and has $(p+1)\times(p+1)(p+2)(p+3)/6$ degrees of freedom.

Isoparametric maps are used
 to map the degrees of freedom from the reference element to elements in the mesh on $\mathcal{F}$.
We start with the case $d=1$ and then generalize to higher dimensions.
Let $e$ be an element in the $i^{\mathrm{th}}$ time partition $(t_{i-1},t_i]$ given by $e(t) = [ x_{k-1}(t), x_k(t) ]$,
where $x_{k-1}(t)$ and $x_k(t)$ are polynomials of degree $p$ satisfying $x_{k-1}(t) < x_k(t)$.
The time component of the isoparametric map is given by the affine map
\[
        t       =       (1-\hat{t}) t_{i-1} + \hat{t} \, t_i    =       t_{i-1} + \hat{t} \Delta t_i
\]
for $0 \le \hat{t} \le 1$.
Taking $\Delta x_k(t) = x_k(t) - x_{k-1}(t)$, the time dependent isoparametric map for the element $e(t)$ is
\[
        x(t)    =       (1-\hat{x}) x_{k-1}(t) + \hat{x} \, x_k(t)      =       x_{k-1}(t) + \hat{x} \Delta x_k(t),
\]
with $0 \le \hat{x} \le 1$.

The inverse of the isoparametric map for element $e$ is 
\begin{align*}
        \hat{t} &=      \frac{t-t_{i-1}}{\Delta t_i}    \\
\intertext{and}
        \hat{x}(t)      &=      \frac{x-x_{k-1}(t)}{\Delta x_k(t)},
\end{align*}
and exists whenever $\Delta x_k(t)>0$ for $t$ in $[t_{i-1},t_i]$ and $x$ in $e(t) = [x_{k-1}(t),x_k(t)]$.
Since the inverse of the spatial component is an affine transformation in space,
it holds that the finite element space at any 
fixed time $t$ is a standard 
finite element space of continuous piecewise polynomials of degree $p$ defined over $\Omega$.
We denote this slice of the finite element space by $\mathcal{V}_h^p(t)$.
This property of the finite element space justifies analyzing finite element functions on a time slice --- for example, $\phi(t)$ in $\mathcal{V}_h^p(t)$ ---
in a way that is consistent with the study of finite element methods for autonomous problems.
Note that $\Delta t_i > 0$ and $\Delta x_k(t) > 0$ are always required for the finite element space to be well-defined.

When $d=2$ or $3$, the time component of the isoparametric map is unaffected, but the spatial component
is now given by a vector mapping that is affine in space and a polynomial of degree $p$ in time.
Taking the derivative of the isoparametric map, one obtains the block triangular Jacobian matrix
\[
        J_e(t)  =       \begin{bmatrix} \mathcal{J}_e (t)       &       \Delta t_{i} x_t(t)^T           \\
                                                                        0               &       \Delta t_i 
                                                                                                         \end{bmatrix},
\]
where $\mathcal{J}_e(t)$ represents the $d \times d$ Jacobian matrix of isoparametric map of element $e(t)$,
and $x_t(t)$ is a $d$-vector that is affine in space and a polynomial of degree $p-1$ along the node trajectories $x_k(t)$.
The trajectories traced out by $x(t)$ are called the \emph{characteristic} trajectories of the mesh and
the vector $x_t$ describes the mesh motion of the finite element space.
These characteristic trajectories can be chosen to offset the convection velocity,
by setting $x_t \approx b$, in hopes of attaining greater flexibility in the length largest permissible time step.
As in the case of one spatial dimension, we require $\Delta t_i>0$ and for the spatial mesh
\[
        \mathcal{D}_e(t)        \equiv  | \det( \mathcal{J}_e(t) ) | > 0.
\]
Note that $\mathcal{D}_e(t)$ is proportional to the size of the spatial element $e(t)$,
so this constraint prevents the element from degenerating in time.

As mentioned above, each time slice of the finite element space is a standard finite element discretization on $\Omega$.
Thus, for $\phi$ in $\mathcal{V}_h^p$ we have $\phi(t) \in \mathcal{V}_h^p(t) \subset \mathcal{H}^1(\Omega)$.
Along the characteristic trajectories, the finite element function $\phi$ is a piecewise polynomial of degree $p$,
defined on the partition given by $\{ t_i \}_{i=0}^m$.
That is, the function $\phi(x(t),t)$ is a polynomial of degree $p$ for $t_{i-1} < t \le t_i$ and fixed $x$ in $\Omega$.
Recall that the finite element functions can have discontinuities between the mesh partitions;
define the jump of $\phi$ by
\[
        [\phi](t_i)     =       \lim_{\delta \rightarrow 0^+} \phi(t_i+\delta) - \phi(t_i-\delta)               \equiv  \phi(t_{i^+}) - \phi(t_{i^-}).
\]
To uniquely define the finite element functions at these discontinuities, following Dupont \cite{DUPONT82},
we require the jump to be orthogonal to the finite element space at the beginning of the new partition:
\begin{equation}        \label{jump orthogonality}
        \bigip{ [ \phi ](t_i) }{\chi}   =       0,
\end{equation}
for all $\chi$ in $\mathcal{V}_h^p(t_{i^+})$.

\section{Preliminary Results}\label{sec3}

Multi-index notation is used to represent spatial derivatives, 
but time and characteristic derivatives will not follow this convention.
The $\mathcal{H}^k(\Omega)$ semi-norm and norm follow conventional notation and we write
\begin{align*}
        \snorm{v}_k     &=      \left(\sum_{|\alpha| = k} \ip{ D_{\alpha} v }{ D_{\alpha} v } \right)^{1/2}     \\
\intertext{and}
        \norm{v}_k      &=      \left(\sum_{|\alpha| \le k} \ip{ D_{\alpha} v }{ D_{\alpha} v } \right)^{1/2}.
\end{align*}
Following Dupont \cite{DUPONT82}, a mesh-dependent semi-norm is defined that allows us to prove our symmetric error estimate,
\[
        \norm{v}_{(-1,\mathcal{V}_h^p(t))}      =       \sup_{\substack{ \chi \in \mathcal{V}_h^p(t) \\ \chi \neq 0}}   \frac{ \left| \ip{ v }{\chi} \right|}{ \norm{\chi}_1}.
\]
We also use the infinity norm, $\norm{v}_{\infty}       =       \max_{x\in\Omega} \snorm{v(x)}$.

Let $\mathcal{Q}_{\mathrm{ref}}$ represent a \emph{reference} quadrature rule, which is defined on the time reference element $[0,1]$
as the interpolatory quadrature rule with knots at a given set of time collocation nodes $\{ \hat{t}_j \}_{j=1}^p$.
It is required that the knots of the $\mathcal{Q}_{\mathrm{ref}}$ satisfy
\begin{equation}        \label{unique collocation}
        0       <       \hat{t}_1 < \ldots < \hat{t}_p  \le     1,
\end{equation}
and that the weights $\{w_j\}$ are all positive for $j=1,\ldots,p$.
For convenience, let $\hat{t}_0 = 0$ be coincident with a degree of freedom on the reference element.
Using the isoparametric maps, this quadrature rule can be applied to the time partitions:
\[
        \mathcal{Q}_i( v )      \equiv  \sum_{j=1}^p w_j v ( t_{i-1} + \hat{t}_j \Delta t_i )
                                        \approx \frac{1}{\Delta t_i} \int_{t_{i-1}}^{t_i} v(t)\ dt.
\]

We now introduce a space-time shape regularity constraint for the moving finite elements,
that controls the time evolution of the spatial elements 
and prevents degenerate elements.
Fix $e$ to be an element in the time partition with $t_{i-1} \le t \le t_i$.
Then, the Jacobian matrix at time $t$ can be represented as
\begin{equation}        \label{jacobian evolution}
        \mathcal{J}_e(t)        =       \big( \mathcal{R}_{e}(t) + \Delta t_i \mathcal{H}_e(t) \big) \mathcal{J}_e( t_{i-1^+} ),
\end{equation}
for some orthogonal \emph{rotation} matrix, $\mathcal{R}_{e}(t)$, and \emph{evolution} matrix, $\mathcal{H}_e(t)$.
The matrix $\mathcal{R}_{e}+\Delta t_i \mathcal{H}_e$ is constrained to have polynomial entries of degree at most $p$ throughout the time partition.
 As the name suggests, the matrix $\mathcal{R}_{e}(t)$ describes the element rotation in time,
 and the evolution matrix describes the deformation of the shape of the element.
Since the trajectories of the spatial nodes are restricted to polynomial paths of degree $p$,
elements do not rotate perfectly and more of a twisting action is observed;
the evolution matrix necessarily reflects these deformations.
If an element is merely translated in time, without rotation or changing shape,
 then the Jacobian matrix, $\mathcal{J}_e(t)$, remains unchanged.

Let $\rho(\cdot)$ represent the spectral radius for $d \times d$ matrices.
For space-time regularity, it is assumed that the evolution matrix $\mathcal{H}_e$ has a uniformly bounded spectral radius throughout the time step;
namely, there exists some positive constant $\mu$ that does not depend on $e$ or $t$ such that
\begin{equation}        \label{space time regularity}
        \rho \big( \mathcal{H}_e(t) \big)       \le \mu.
\end{equation}
This can be interpreted as bounding the relative change in shape and size of the element over time.

Assuming a non-degenerate finite element space and the space-time shape regularity bound \eref{space time regularity}, it follows that
\begin{equation}        \label{spectral radius bound}
        \rho \big( \mathcal{J}_e(t) \mathcal{J}_e^{-1}(t_{i-1^+}) \big) =       \rho \big( \mathcal{R}_{e}(t) + \Delta t_i \mathcal{H}_e(t) \big)
                                                                                        \le     1 + \mu \Delta t_i
\end{equation}
and, for $\tilde{c}_{\mu,d} = [ (1+\mu \Delta t_i)^d - 1] / \Delta t_i = \mathcal{O}(1)$ and $\Delta t_i \le 1/2\tilde{c}_{\mu,d}$,
\begin{multline}        \label{determinant bound}
        (1-\tilde{c}_{\mu,d}\Delta t_i) \le     (1- \mu \Delta t_i)^d   \le     \frac{\mathcal{D}_e(t)}{\mathcal{D}_e(t_{i-1^+})}               \\
                        =       \det \big(\mathcal{J}_e(t) \mathcal{J}_e^{-1}(t_{i-1^+}) \big)  \le     (1 + \mu \Delta t_i)^d  \le     (1+ \tilde{c}_{\mu,d}\Delta t_i),
\end{multline}
since $\mathcal{R}_{e}$ is an orthogonal matrix.

Let $\phi$ be a function in the finite element space $\mathcal{V}_h^p(t)$ for some $t$ in the time partition $(t_{i-1},t_i]$.
We \emph{shift} $\phi$ onto the mesh of $\mathcal{V}_h^p(t_{i-1^+})$, at the beginning of the time partition by replacing
the basis functions of $\mathcal{V}_h^p(t)$ with with their corresponding basis functions in $\mathcal{V}_h^p(t_{i-1^+})$,
while preserving the basis coefficients.
Formally, this operation can be defined by an element-wise composition of the inverse of the affine spatial isoparametric maps for the elements in the mesh at time $t$,
which is well-defined for non-degenerate meshes,
with the affine spatial isoparametric maps for the elements at the beginning of the time step $t_{i-1^+}$.
The following lemma establishes the relationship between the space-time shape regularity constraint \eref{space time regularity}
and the continuity of this shift operation.

\begin{lemma}[Shift Lemma]      \label{shift lemma}
Let $\phi,\chi \in \mathcal{V}_h^p(t)$ and $\tilde{\phi},\tilde{\chi} \in \mathcal{V}_h^p(t_{i-1^+})$ represent a pair of finite element functions and their shifts, respectively,
on a non-degenerate time partition of the mesh that satisfies \eref{space time regularity} on each element.
If $\Delta t_i \le 1/2\tilde{c}_{\mu,d}$, as defined in \eref{determinant bound}, then there exists a positive constant $C_{\mu,d}$ such that
\begin{align}
        \Big| \bigip{\phi}{\chi} - \bigip{\tilde{\phi}}{\tilde{\chi}} \Big|     &\le    C_{\mu,d} \Delta t_i \frac{\bignorm{\tilde{\phi}}_0^2   +       \bignorm{\tilde{\chi}}_0^2}{2}, \label{IP1}\\
        \Big|   \bignorm{\phi}_0^2 - \bignorm{\tilde{\phi}}_0^2 \Big|   &\le    C_{\mu,d} \Delta t_i \bignorm{\tilde{\phi}}_0^2,                                                                                \label{L2 bound}\\
        \Big|   \bignorm{\phi}_1^2 - \bignorm{\tilde{\phi}}_1^2 \Big|   &\le    C_{\mu,d} \Delta t_i \bignorm{\tilde{\phi}}_1^2.                                                                                \label{H1 bound}
\end{align}
\end{lemma}

\begin{proof}
The proof follows from an element-wise change of variables and using the space-time shape regularity constraint \eref{space time regularity}
to uniformly bound the relative change in element size by bound \eref{determinant bound}.
For bounding the difference of a function and its shift in the $\mathcal{H}^1$-norm,
the Jacobian matrices of the element-wise transformations multiply the finite element function, following from the chain rule.
These Jacobian matrices are bounded by \eref{spectral radius bound}.
See \cite{METTITHESIS} for a more detailed proof.
\end{proof}

We conclude this section with the following discrete Gr\"onwall inequality.
\begin{lemma}[Discrete Gr\"onwall Inequality]   \label{discrete gronwall}
Let $\Delta t_i > 0$ and $\alpha_i, \gamma_i, \theta_i, q_i \ge 0$, for $1 \le i \le m$, with $\theta_i\Delta t_i \le \frac{1}{2}$ and $\theta = \max_i \theta_i$.
Then, if
\[
        \frac{q_i - q_{i-1}}{\Delta t_i} + \gamma_i     \le     \alpha_i + \theta_i ( q_i + q_{i-1} ),
\]
there exists a positive constant $C_{\theta}$ such that
\[
        \max_{1 \le i \le m} q_i + \sum_{i=1}^m \gamma_i \Delta t_i             \le C_{\theta} \left\{ q_0 + \sum_{i=1}^m \alpha_i \Delta t_i \right\}.
\]
\end{lemma}

This theorem comes directly from \cite{SANTOSTHESIS} and the proof can be found therein.
Its argument primarily follows the proof of the standard discrete Gr\"onwall lemma with minor modifications.

\section{A space-time moving finite element method}\label{sec4}

We now discretize the weak formulation of the differential equation \eref{variational form diff eqn}.
Given the mesh velocity, $x_t$, and collocation nodes, $\{ t_{i,j} \}$, find $u_h$ in $\mathcal{V}_h^p$ such that for each $t_{i,j}$ and all $\chi$ in $\mathcal{V}_h^p(t_{i,j})$,
the finite element solution satisfies
\begin{align}
        \bigip{\partial_\tau u_h(t_{i,j})}{\chi}     +       \mathcal{A}_\tau\bigip{t_{i,j};u_h}{\chi}       &=      \bigip{f(t_{i,j})}{\chi}        +       \bigbip{g(t_{i,j})}{\chi}       \label{fe form diff eqn}        \\
\intertext{for $i=1,\ldots,m$ and $j=1,\ldots,p$, and when $t=0$,}
                                                        \bigip{u_h(\cdot,0)}{\chi}      &=      \bigip{u_0}{\chi}.      \nonumber
\end{align}
Each time partition is coupled to the previous time partition by the jump orthogonality condition \eref{jump orthogonality},
which states that the finite element solution must satisfy
\[
        \ip{u_h(t_{i^+})}{\chi} =       \ip{u_h(t_{i^-})}{\chi},
\]
for all $\chi$ in $\mathcal{V}_h^p(t_{i^+})$, at each time step $i=1,\ldots,m-1$.

This formulation effectively solves for the solution on each time partition sequentially.
The mesh motion is assumed to be pre-computed for the time partition and no specific mesh motion is prescribed.
As a result, this analysis encompasses many mesh moving strategies 
including the method of characteristics ($x_t = b$) and non-moving meshes ($x_t \equiv 0$).
Since the mesh motion is assumed to be known before computing the solution, however, strategies for moving 
the mesh cannot depend on the values of the computed solution,
unless values from previous time partitions or predictor-corrector schemes are used.

An important feature of this method is that the time collocation nodes, where \eref{fe form diff eqn} is imposed,
need not coincide with the time basis nodes. Here the collocation nodes will be the $p$ nodes of quadrature  formula 
$\mathcal{Q}$ of order at least $2p-1$. For any polynomial $v\in\Re^{2p-1}$, we assume that $\mathcal{Q}$ satisfies
\begin{equation}\label{quad_rule}
\int_0^1v\,dx=\mathcal{Q}(v)-C_pv^{(2p-1)}
\end{equation}
where $C_p\ge 0$. $C_p=0$ when $\mathcal{Q}$ is classical Gaussian quadrature, since its order is $2p$, while $C_p>0$ for Gauss-Radau
quadrature \cite{NOTARIS}. $C_p$ is also positive for the one parameter family of quadrature rules of order $2p-1$ that interpolates between
the Gauss and Gauss-Radau rules. In the special case $p=1$, the Gauss rule corresponds to an integration scheme analogous to
the Crank-Nicolson method; the Gauss-Radau rule is analogous to the first backward difference formula. The family of rules that
interpolate between the two corresponds to the family of so-called $\theta$-methods.

Let $\{ t_{i,j} \}$ represent the collocation nodes and $\{ \zeta_{i,j} \}$ represent the time basis nodes,
and let $\phi \in \mathcal{V}_h^p$.
At the collocation node $t=t_{i,j}$, let $\chi \in \mathcal{V}_h^p(t_{i,j})$ and $\tilde{\phi}(t) \in \mathcal{V}_h^p(t_{i,j})$
be the shift of $\phi(t)$ onto the mesh at time $t_{i,j}$.
Then, the finite element solution is determined by
\begin{multline*}
        \sum_{k=0}^p \Big\{        \frac{1}{\Delta t_i} \beta_k'(\hat{t}_{j})  \ip{\tilde{\phi}(\zeta_{i,k})}{\chi} 
                        \\\mbox{}                   + \beta_k(\hat{t}_j) \big[ \ip{ a \nabla \tilde{\phi}(\zeta_{i,k})}{\nabla \chi}
                                                                                +\ip{ (b-x_t) \cdot \nabla \tilde{\phi}(\zeta_{i,k})}{\chi}
                                                                        +\ip{ c\tilde{\phi}(\zeta_{i,k})}{\chi} \big] \Big\}
                \\      =       \ip{f}{\chi} + \bip{g}{\chi},
\end{multline*}
where $\beta_j$ represents the time basis functions on the reference element and $\hat{t}_j$ represents the collocation nodes on the reference element,
$j = 1, \ldots, p$.
To ensure the existence and uniqueness of the finite element solution for sufficiently small $\Delta t_i$,
we show that the following necessary condition is satisfied.

\begin{lemma}   \label{invertibility of the derivative matrix}
Let $\{ {\beta}_j \}_{j=0,\ldots, p}$ be the Lagrange basis for polynomials of degree $p$ with nodes
\(
        0 =     \hat{\zeta}_0   < \hat{\zeta}_1 <       \cdots  < \hat{\zeta}_p         \le 1.
\)
Then, the matrix defined by $B \equiv [ \beta'_k(\hat{t}_j) ]_{j,k}$ is invertible,
where $\hat{t}_j$ is a strictly ordered partition of $(0,1]$ and $1 \le j,k \le p$. Furthermore $\norm{B}$ and $\norm{B^{-1}}$
are independent of the time steps, space discretization, and details of the pde.
\end{lemma}

\begin{proof}
Let $v \in \Re^p$ be chosen to satisfy $Bv = 0$.
Then, the polynomial of degree $p$ defined by
\[
        v(\hat{t})      \equiv  \sum_{j=1}^p {\beta}_j(\hat{t}) v_j
\]
satisfies $v'(\hat{t}_j) = 0$ for $j=1,\ldots,p$.
Accordingly, the derivative $v'(\hat{t})$ is identically zero as it has $p$ distinct roots.
Thus, the polynomial $v(\hat{t})$ is constant and ${\beta}_j(0) = 0$ for $j=1,\ldots,p$, which implies that $v(\hat{t}) \equiv 0$.
Equivalently, the vector $v$ must be trivial.
Since all calculations take place on the reference interval, $\norm{B}$ and $\norm{B^{-1}}$ can only depend on the
quadrature rule $\mathcal{Q}$, the basis functions on the reference element, and the choice of norm.
\end{proof}

From lemma \ref{invertibility of the derivative matrix}, the linear system corresponding to the finite element formulation is nonsingular,
which proves the existence and uniqueness of the finite element solution exists and is unique when $\Delta t_i$ is sufficiently small.

We now present a local Gr\"onwall lemma that will be used to bound the maximum error of the finite element solution in the $\mathcal{L}_2$-norm over each time partition.

\begin{lemma}[Local Gr\"onwall Inequality]      \label{local gronwall}
Let $\mathcal{Q}_{i}$ be the interpolatory quadrature rule with positive weights and $p$ distinct collocation nodes $\{ t_{i,j} \}$ on the time partition $(t_{i-1},t_i]$.
Suppose the space-time mesh on the time partition satisfies the regularity constraint \eref{space time regularity} at the collocation nodes, where $1 \le i \le m$
and that there exists a positive constant $\kappa$ such that
\begin{equation}        \label{mesh alignment local lemma}
        \norm{b - x_t}_\infty   \le     \kappa
\end{equation}
If $\Delta t_i \le1/2\tilde{c}_{\mu,d}$, as defined in \eref{determinant bound},
and functions $\phi$ in $\mathcal{V}_h^p$ and $\eta$ in the solution space satisfy
\begin{equation}        \label{collocation constraint}
        \bigip{\phi_{\tau}(t_{i,j})}{\chi}      +       \mathcal{A}_\tau\bigip{t_{i,j};\phi}{\chi}      
                        =       \bigip{\eta_{\tau}(t_{i,j})}{\chi}      +       \mathcal{A}_\tau\bigip{t_{i,j};\eta}{\chi}      
\end{equation}
for all $\chi$ in $\mathcal{V}_h^p(t_{i,j})$ at time each collocation node $t_{i,j}$,
then, there exists a constant such that
\[
        \max_{1 \le j \le p} \bignorm{ \phi(t_{i,j}) }_0^2      \le
                C \left\{       \bignorm{ \phi(t_{i-1^+})}_0^2 + \Delta t_i \mathcal{Q}_i \left( \norm{\eta_{\tau}}_{(-1,\mathcal{V}_h^p(\cdot))}^2 + \norm{\eta}_1^2 + \norm{\phi}_1^2 \right) \right\},
\]
where $C$ depends on $\kappa, \mu, d, p$, and the differential equation.
\end{lemma}

\begin{proof}
To simplify notation, the time partition index, $i$, is assumed to be fixed and we let $\phi_j \equiv \phi(t_{i,j})$.
Define $\tilde{\phi}^{(j)}_\ell$ in $\mathcal{V}_h^p(t_{i,j})$  to be the finite element function with the same basis coefficients as $\phi_\ell$
multiplying the basis functions for time $t_{i,j}$.
The function $\tilde{\phi}^{(j)}_\ell$ is, therefore, the shift of $\phi(t_{i,\ell})$ onto the mesh at time $t_{i,j}$, and $\tilde{\phi}^{(j)}_j = \phi_j$.

Let $k$ index the collocation node where $\phi$ attains its maximal $\mathcal{L}_2$ norm:
$\norm{\phi_k}_0 = \max_{1 \le j \le p} \norm{\phi_j}_0$.
Choose $\chi = \tilde{\phi}^{(j)}_k$ in $\mathcal{V}_h^p(t_{i,j})$ for equation \eref{collocation constraint}, $j=1,\ldots,p$.
Use the time basis expansion of the characteristic derivative to get
\[
        \frac{1}{\Delta t_{i}} \sum_{\ell=0}^p {\beta}_\ell'(\hat{t}_j) \bigip{\tilde{\phi}_{\ell}^{(j)}}{\tilde{\phi}^{(j)}_k} + \mathcal{A}_\tau\bigip{t_{i,j};\phi_j}{\tilde{\phi}^{(j)}_k}
                =       \bigip{\eta_{j,\tau}}{\tilde{\phi}^{(j)}_k} + \mathcal{A}_\tau\bigip{t_{i,j};\eta_j}{\tilde{\phi}^{(j)}_k},
\]
for $j=1,\ldots,p$, where $\beta_j$ and $\hat{t}_j$ represents the time basis functions and the collocation nodes on the reference element.
Equivalently,
\begin{multline}        \label{coll constraint v2}
        \sum_{\ell=1}^p {\beta}_\ell'(\hat{t}_j) \Bigip{ \tilde{\phi}^{(j)}_\ell}{\tilde{\phi}^{(j)}_k}
                =       - {\beta}'_0(\hat{t}_j) \Bigip{\tilde{\phi}^{(j)}_0}{\tilde{\phi}^{(j)}_k}      \\
                \mbox{} + \Delta t_{i} \left\{ \Bigip{ \eta_{j,\tau} }{\tilde{\phi}^{(j)}_k} + \mathcal{A}_\tau\bigip{t_{i,j};\eta_j - \phi_j}{\tilde{\phi}^{(j)}_k} \right\}.
\end{multline}
By lemma \ref{invertibility of the derivative matrix}, there exists a linear combination of the derivatives of the time basis functions such that
\[
        \sum_{j=1}^p \alpha_j \sum_{\ell=1}^p {\beta}_\ell'(\hat{t}_{j}) v_\ell = v_k,
\]
for $1\le k \le p$ and any $v = [v_j]_j \in \Re^p$.
Thus, we take this linear combination of the equation \eref{collocation constraint}
so that the left side simplifies and can be bounded using lemma \ref{shift lemma}:
\begin{multline}        \label{left derivative bound}
        \sum_{j=1}^p \alpha_j \sum_{\ell=1}^p {\beta}_\ell'(\hat{t}_{j}) \Bigip{ \tilde{\phi}^{(j)}_\ell}{\tilde{\phi}^{(j)}_k}                 \\
                \ge     \sum_{j=1}^p \alpha_j \sum_{\ell=1}^p {\beta}_\ell'(\hat{t}_j) \Bigip{ \tilde{\phi}^{(0)}_\ell}{\tilde{\phi}^{(0)}_k}
                        - \frac{1}{2} C_{\mu} \Delta t_i \sum_{j=1}^p \left| \alpha_j \right| \sum_{\ell=1}^p \left| {\beta}_\ell'(\hat{t}_j) \right| 
                                                \left( \bignorm{ \tilde{\phi}^{(0)}_\ell}_0^2 + \bignorm{\tilde{\phi}^{(0)}_k}_0^2 \right)              \\
                \ge     \bignorm{\tilde{\phi}^{(0)}_k}_0^2      -       \hat{C}_{\mu,d,p} \Delta t_i \sum_{j=1}^p \bignorm{\tilde{\phi}^{(0)}_j}^2_0
                \ge     \bignorm{\tilde{\phi}^{(0)}_k}_0^2      -       \hat{C}'_{\mu,d,p} \Delta t_{i} \sum_{j=1}^p \norm{\phi_j}^2_0.
\end{multline}

To bound the right side, choose $\delta>0$ to be sufficiently small, say $\delta < 1/2$, and use the shift lemma (lemma \ref{shift lemma}) to get
\begin{equation}        \label{initial condition bound}
        - \sum_{j=1}^p \alpha_j {\beta}'_0(\hat{t}_j) \Bigip{\tilde{\phi}^{(j)}_0}{\tilde{\phi}^{(j)}_k}
                \le     \hat{C}''_{\mu,d,p} \bignorm{\phi_0}_0^2        +       \delta \bignorm{\tilde{\phi}^{(0)}_k}_0^2.
\end{equation}
Furthermore, we use the mesh dependent negative norm and the shift lemma to bound
\begin{equation}
        \Bigip{ \eta_{j,\tau} }{\tilde{\phi}^{(j)}_k}   \le     \frac{1}{2} \left( \norm{\eta_{j,\tau}}_{(-1,\mathcal{V}_h^p(t_{i,j}))}^2 + C_{\mu,d,p} \norm{\phi_k}_1^2 \right)
\end{equation}
and
\begin{align}
        \mathcal{A}_\tau\Bigip{t_{i,j};\eta_j - \phi_j}{\tilde{\phi}^{(j)}_k}   
                                                                                &\le    C_{\mathcal{A},\kappa}\left({\norm{\eta_j}_1^2 + \norm{\phi_j}_1^2}
                                                                                                        + \bignorm{\tilde{\phi}^{(j)}_k}_1^2 \right)            \nonumber \\
                                                                                &\le    C_{\mathcal{A},\kappa}\left({\norm{\eta_j}_1^2 + \norm{\phi_j}_1^2}
                                                                                                + C_{\mu,d,p} \norm{\phi_k}_1^2 \right).                \label{energy term bound}
\end{align}

From \eref{left derivative bound}--\eref{energy term bound}, we get
\begin{equation}        \label{shifted bound}
        \bignorm{\tilde{\phi}^{(0)}_k}  _0^2    \le     C_{\mu,d,p} \Big\{ \norm{\phi(t_{i-1^+})}_0^2 + C_{\mathcal{A},\kappa} \Delta t_i
                         \sum_{j=1}^p \Big( \norm{\eta_{j,\tau}}_{(-1,\mathcal{V}_h^p(t_{i,j}))}^2 + \norm{\eta_j}_1^2 + \norm{\phi_j}_1^2 \Big) \Big\}.
\end{equation}
Since the quadrature weights are positive, there is a constant $C_{\mathcal{Q}}>0$ such that
\begin{equation}        \label{sum and quadrature comparison}
         \sum_{j=1}^p \Big( \norm{\eta_{j,\tau}}_{(-1,\mathcal{V}_h^p(t_{i,j}))}^2 + \norm{\eta_j}_1^2 + \norm{\phi_j}_1^2 \Big)
                \le C_{\mathcal{Q}} \mathcal{Q}_i \Big(\norm{\eta_{\tau}(\cdot)}_{(-1,\mathcal{V}_h^p(\cdot))}^2 + \norm{\eta(\cdot)}_1^2 + \norm{\phi(\cdot)}_1^2 \Big).
\end{equation}  
By the shift lemma, $\bignorm{\tilde{\phi}^{(0)}_k} \le C_{\mu, d,p}\norm{\phi_k}_0 = C_{\mu, d,p}\max_{1\le j \le p} \norm{\phi_j}_0$,
and combining the bounds \eref{shifted bound}--\eref{sum and quadrature comparison} yields the desired result.
\end{proof}

The mesh-dependent energy norm used in our symmetric error estimate 
is given by
\[
        \enorm{ u }^2   =       \max_{\substack{1 \le i \le m\\1 \le j \le p}}  \bignorm{ u (t_{i,j}) }_0^2
                                        + \sum_{i=1}^m \Delta t_i \mathcal{Q}_i \left( \norm{u_{t_{i,j}au}(\cdot)}_{(-1,\mathcal{V}_h^p(\cdot))}^2 + \norm{u(\cdot)}_1^2 \right).
\]

The main result of this paper is the quasi-optimal error bound for the finite element solution determined by \eref{fe form diff eqn}.

\begin{theorem}\label{main_theorem}
Suppose that $\mathcal{V}_h^p$ is a finite element space with a non-degenerate mesh
and time collocation nodes that satisfy \eref{quad_rule}.
Furthermore, assume that there exist positive constants $\mu$ and $\kappa$ such that at each collocation node
\begin{align}
        \rho \Big(\mathcal{H}_e(t_{i,j})\Big)     &\le    \mu,    \label{spacetime regularity}    \\
\intertext{and}
        \norm{b - x_t}_\infty                           &\le    \kappa. \label{characteristic regularity}
\end{align}
Then, if $\Delta t = \max_{1\le i \le m} \Delta t_i$ is sufficiently small, there exists a positive constant $C$ such that the finite element solution satisfies
\begin{equation}        \label{a priori stmfem}
        \enorm{u - u_h} \le     C \inf_{\chi \in \mathcal{V}_h^p} \enorm{u - \chi},
\end{equation}
where $C$ depends on $\mu, \kappa, d, p$, and the differential equation.
\end{theorem}

\begin{proof}
From the Galerkin orthogonalities, at each collocation node
\[
        \bigip{\partial_\tau u_h(t_{i,j})}{\chi}        +       \mathcal{A}_\tau\bigip{t_{i,j};u_h}{\chi}
                        =       \bigip{\partial_\tau u(t_{i,j})}{\chi} + \mathcal{A}_\tau\bigip{t_{i,j};u}{\chi},
\]
for all $\chi$ in $\mathcal{V}_h^p(t_{i,j})$ with $i=1,\ldots,m$ and $j=1,\ldots,p$.
Let $\psi\in\mathcal{V}_h^p$ and define $\phi \equiv u_h - \psi$ in $\mathcal{V}_h^p$ and $\eta \equiv u - \psi$.
Re-write the statement of the Galerkin orthogonality as
\begin{equation}        \label{theorem galerkin}
        \ip{ \partial_\tau\phi(t_{i,j})}{\chi}  +       \mathcal{A}_\tau\bigip{t_{i,j};\phi}{\chi}      =       \ip{\partial_\tau\eta(t_{i,j})}{\chi} + \mathcal{A}_\tau\bigip{t_{i,j};\eta}{\chi},
\end{equation}
for all $\chi$ in $\mathcal{V}_h^p(t_{i,j})$.
Using equation \eref{theorem galerkin} we will show
\[
        \enorm{\phi}    \le     C \enorm{\eta}
\]
and use the triangle inequality to obtain the sought after bound \eref{a priori stmfem}.

Fix $i$ and $j$ and choose $\chi = \phi(t_{i,j})$ so that \eref{theorem galerkin} gives
\begin{equation}        \label{coll ortho}
        \ip{\phi_{\tau}}{\phi}  +       \mathcal{A}_\tau\ip{\phi}{\phi} =       \ip{\eta_{\tau}}{\phi}  +       \mathcal{A}_\tau\ip{\eta}{\phi},
\end{equation}
at time $t=t_{i,j}$.
The bound at this collocation node comes from bounding the terms in \eref{coll ortho} individually.
For the first term on the left, the shift lemma gives the bound
\begin{equation}        \label{characteristic identity}
        \ip{\phi_{\tau}}{\phi}  =       \frac{1}{2} \partial_\tau\norm{\phi}_0^2        
                                        \ge     \frac12 \frac{d}{dt} \norm{\tilde{\phi}}_0^2
                                                -       C_{\mu,d,p} \max_{1 \le k \le p} \norm{\phi(t_{i,k})}_0^2,
\end{equation}
where $\tilde{\phi}$ is the shift onto the initial mesh of the time partition.
Since the mesh motion satisfies \eref{characteristic regularity},
\begin{equation}
        \mathcal{A}_\tau\ip{\phi}{\phi} 
                \ge \bar{a} \snorm{\phi}_1^2 -\kappa \snorm{\phi}_1 \norm{\phi}_0 + \bar{c} \norm{\phi}_0^2
                \ge     C_{\mathcal{A}} \norm{\phi}_1^2 - C_{\mathcal{A},\kappa} \norm{\phi}_0^2.
\end{equation}
Now, choose $\varepsilon > 0$ to be sufficiently small and bound
\begin{equation}
        \ip{\eta_{\tau}}{\phi}  \le     C \norm{\eta_{\tau}}_{(-1,\mathcal{V}_h^p(t_{i,j}))}^2 + \varepsilon \norm{\phi}_1^2
\end{equation}
and, since \eref{characteristic regularity} holds,
\begin{equation}        \label{right energy form bound}
        \mathcal{A}_\tau\ip{\eta}{\phi}                 \le     C_{\mathcal{A},\kappa} \norm{\eta}_1^2 + \varepsilon \norm{\phi}_1^2.
\end{equation}
Hence, from \eref{coll ortho}--\eref{right energy form bound}, it is true at $t=t_{i,j}$ that
\begin{equation}        \label{coll bound}
        \frac{1}{2}\frac{d}{dt}\norm{\tilde{\phi}}_0^2  +       C_{\mathcal{A}} \norm{\phi}_1^2 
                \le     C_{\mathcal{A},\kappa,\mu,d,p} \left\{ \norm{\eta_{\tau}}_{(-1,\mathcal{V}_h^p(t_{i,j}))}^2 + \norm{\eta}_1^2 + \max_{1 \le k \le p}\norm{\phi(t_{i,k})}_0^2 \right\}.
\end{equation}

Using the local quadrature rule $\mathcal{Q}_i$, with positive weights, to aggregate \eref{coll bound} over the time partition, we recover
\begin{equation}        \label{agg coll bound}
        \mathcal{Q}_i\Big(\frac{d}{dt}\norm{\tilde{\phi}}_0^2\Big)      +       C_{\mathcal{A}} \mathcal{Q}_i\big(\norm{\phi}_1^2       \big)
                \le     C_{\mathcal{A},\kappa,\mu,d,p} \left\{ \mathcal{Q}_i\big(\norm{\eta_{\tau}}_{(-1,\mathcal{V}_h^p(t_{i,j}))}^2 + \norm{\eta}_1^2\big)
                                                                         + \max_{1 \le k \le p}\norm{\phi(t_{i,k})}_0^2 \right\}.
\end{equation}
The collocation nodes satisfy \eref{quad_rule},
which implies that
\begin{equation}        \label{quad time derivative}
        \mathcal{Q}_i\Big(\frac{d}{dt}\norm{\tilde{\phi}}_0^2\Big)      \ge     \frac{1}{\Delta t_i} \int_{t_{i-1}}^{t_i} \frac{d}{dt}\norm{\tilde{\phi}}_0^2\ dt
                                                                                                =       \frac{\norm{\tilde{\phi}(t_{i^-})}_0^2-\norm{\tilde{\phi}(t_{i-1^+})}_0^2}{\Delta t_i}.
\end{equation}
From bound \eref{quad time derivative}, the shift lemma, and the local Gr\"onwall inequality, we may write
\begin{multline}        \label{pre gronwall bound}
        \frac{\norm{{\phi}(t_{i^-})}_0^2-\norm{{\phi}(t_{i-1^+})}_0^2}{\Delta t_i}      
                +       C_{\mathcal{A}} \mathcal{Q}_i\big(\norm{\phi}_1^2       \big)
        \\      \le     C_{\mathcal{A},\kappa,\mu,d,p} \left\{ \mathcal{Q}_i\big(\norm{\eta_{\tau}}_{(-1,\mathcal{V}_h^p(t_{i,j}))}^2 + \norm{\eta}_1^2\big) 
                                +       \norm{{\phi}(t_{i^-})}_0^2+\norm{{\phi}(t_{i-1^+})}_0^2 \right\}.
\end{multline}
We now use the discrete Gr\"onwall inequality to show
\begin{equation}        \label{partition gronwall}
        \max_{0 \le i \le m} \norm{\phi(t_{i^-})}_0^2 + \sum_{i=1}^m \Delta t_i \mathcal{Q}_i( \norm{\phi}_1^2 )        \le     C \left\{ \norm{\phi(0)}_0^2 + \enorm{\eta}^2 \right\}.
\end{equation}
Since an $\mathcal{L}_2$-projection is used for the initial condition,
\begin{equation}
        \norm{\phi(0)}_0        \le \norm{\eta(0)}_0    \le \enorm{\eta}.
\end{equation}
Furthermore, at each collocation node, $t_{i,j}$,
\[
        \ip{\partial_\tau\phi}{\chi}    =       \ip{\partial_\tau\eta}{\chi} + \mathcal{A}_\tau\ip{\eta}{\chi} - \mathcal{A}_\tau\ip{\phi}{\chi}
\]
for all $\chi$ in $\mathcal{V}_h^p(t_{i,j})$, which implies that
\begin{equation}        \label{negative norm bound}
        \norm{\partial_\tau\phi}_{(-1,\mathcal{V}_h^p(t_{i,j}))}
                \le     C_{\mathcal{A},\kappa} \left\{ \norm{\partial_\tau\eta}_{(-1,\mathcal{V}_h^p(t_{i,j}))} + \norm{\eta}_1 + \norm{\phi}_1 \right\}.
\end{equation}
From \eref{partition gronwall}--\eref{negative norm bound},
\[
                \max_{0 \le i \le m} \norm{\phi(t_{i^-})}_0^2 + \sum_{i=1}^m \Delta t_i \mathcal{Q}\big( \norm{\phi}_1^2 + \norm{\phi}_{(-1,\mathcal{V}_h^p(\cdot))}^2 \big)
                                        \le     C \enorm{\eta}^2.
\]
All that is needed to conclude the proof is the local Gr\"onwall inequality once more to show 
\begin{equation}
        \max_{1\le j \le p} \norm{\phi(t_{i,j})}_0^2    \le C \left\{ \norm{ \phi(t_{i-1^+})}_0^2 + 
                                                                                        \Delta t_i \mathcal{Q}_i \left( \norm{\eta_{\tau}}_{(-1,\mathcal{V}_h^p(\cdot))}^2 + \norm{\eta}_1^2 + \norm{\phi}_1^2 \right) \right\}
                                                                        \le C \enorm{\eta}^2
\end{equation}
for $i=1,\ldots,m$.
Thus, we have $\enorm{\phi} \le C \enorm{\eta}$, as desired.
\end{proof}

\section{Discussion}\label{sec5}

In the case $p=1$, Theorem \ref{main_theorem} improves on the results given in \cite{BANKSANTOS} in several
small but significant ways. First, the energy norm $\enorm{\cdot}$ employs the characteristic space-time derivative 
rather than the time derivative, an improvement first used in \cite{LIU_ETAL}. Theorem \ref{main_theorem}
also unifies proofs for Crank-Nicolson and first backward difference approaches under one umbrella, and
extends the theory to the family of first order $\theta$-methods.  
For the case $p>1$, we believe our result to be new.

The family of methods covered by Theorem \ref{main_theorem} corresponds to a family of $p$-stage fully 
implicit Runge-Kutta methods. Except for the case $p=1$, these methods are rarely used in practice.
Being fully implicit, all degrees of freedom from all stages become coupled, resulting in a system of
equations of order $Np$ to be solved, where $N$ is the number of degrees of freedom in the space dimensions.
When $p>1$, one normally moves to diagonally implicit Runge-Kutta methods, where degrees of freedom
associated with stage $r$ depend only on degrees of freedom from stages $0,1,\dots,r-1$. This gives the
system of equations a block triangular shape, yielding $p$ similar systems of order $N$ to be solved in each time step,
rather than one system of order $Np$. Indeed, the example that inspired this work is the TR-BDF2 method
\cite{A30,A30a,STRANG,SHAMPINE}, which in this context is a second order two stage diagonally implicit
Runge-Kutta method. More informally, it consists of a first half-step using the trapezoid rule (Crank-Nicolson)
followed by a half-step using the second backward difference formula. Unfortunately, thus far we have been
unable to prove  a symmetric error estimate like Theorem \ref{main_theorem} for the general case
of $p$-stage diagonally implicit
Runge-Kutta methods. The main issue is that we have not been able to cast such methods in a fully
Galerkin finite element framework. 
We have been successful in analyzing the TR-BDF2 method \cite{PAPER2,METTITHESIS}, but were able to
achieve only a partially symmetric error estimate; in particular, some extra time-truncation terms appear on the
right hand side of the analogue of \eref{a priori stmfem}. This is reminiscent of similar terms that
appeared in the analysis of time discretization schemes in \cite{DUPONT82}.

\bibliography{ref}
\bibliographystyle{siam}
\end{document}